\documentclass[12pt]{article}

\usepackage{ulem}       
\usepackage{amsmath, amsthm, amssymb}

\usepackage[margin=1in]{geometry}  
\usepackage{xspace}

\usepackage{mathrsfs}

\usepackage{mathabx}

 \usepackage{graphicx}  
 \usepackage{color}	
\newtheorem{theorem}{Theorem}
\newtheorem{corollary}{Corollary}
\newtheorem{proposition}{Proposition}

\newtheorem{example}{Example}

\begin{document}
\thispagestyle{empty}

\begin{center}
\textbf{\Large Embedding an Edge-colored $K(a^{(p)};\lambda,\mu )$ into a Hamiltonian Decomposition of $K(a^{(p+r)};\lambda,\mu )$}

\addvspace{\bigskipamount}
Amin Bahmanian\footnote{mzb0004@auburn.edu}, C. A. Rodger\footnote{rodgec1@auburn.edu} \\
Department of Mathematics and Statistics\\
 Auburn University, Auburn, AL\\ USA   36849-5310
\end{center}

\typeout{Abstract}
Let $K(a^{(p)};\lambda,\mu )$  be a graph with $p$ parts, each part having size $a$, in which the multiplicity of each pair of vertices in the same part (in different parts) is  $\lambda$ ($\mu $, respectively). In this paper we consider the following embedding problem:
When can a graph decomposition of $K(a^{(p)};\lambda,\mu )$ be extended to a Hamiltonian decomposition of $K(a^{(p+r)};\lambda,\mu )$ for $r>0$?
A general result is proved, which is then used to solve the embedding problem for all $r\geq  \frac{\lambda}{\mu a}+\frac{p-1}{a-1}$. The problem is also solved when $r$ is as small as possible in two different senses, namely when $r=1$ and when $r=\frac{\lambda}{\mu a}-p+1$.

Keywords. Amalgamations; Detachments; Hamiltonian Decomposition; Edge-coloring; Hamiltonian Cycles; Embedding
\section{Introduction}

Let $G=(V,E)$ be a graph and let $H=\{H_{i}\}_{i\in I}$ be a family of graphs where $H_{i}=(V_{i},E_{i})$.  We say that $G$ has an \textit{$H$-decomposition} if $\{E_{i}:i\in I\}$ partitions $E$ and each $E_{i}$ induces an isomorphic copy of $H_{i}$. Graph decomposition in general has been studied for many classes of graphs. The decomposition of a graph into paths \cite{Tarsi83}, cycles \cite{Sajna 1} or stars \cite{Tarsi81} has been of special interest over the years. Of particular interest is the decomposition of a graph into Hamiltonian cycles; that is a \textit{ Hamiltonian Decomposition}. 
In 1892 Walecki \cite{L} proved the classic result that the complete graph $K_n$ is decomposable into Hamiltonian cycles if and only if $n$ is odd. Laskar and Auerbach \cite{LA} settled the existence of Hamiltonian decomposition of the complete multipartite graph $K_{m,\ldots,m}$. 
Alspach, Gavlas, and \u{S}ajna \cite{Alspach and Gavlas, Sajna 1, Sajna 2} collectively solved the problem of decomposing the complete graph into isomorphic cycles, but the problem remains open for different cycle lengths.

A \textit{k-edge-coloring} of $G$ is a mapping ${\cal K}:E(G)\rightarrow C$, where $C$ is a set of $k$ \textit{colors} (often we use $C=\{1,\ldots,k\}$), and the edges of one color form a \textit{color class}.  
Another challenge is the companion embedding problem:

Let $H=\{H_{i}\}_{i\in I}$ and  $H^*=\{H^*_{j}\}_{j\in J}$ be two families of graphs. Given a graph $G$ with an $H$-decomposition and a graph $G^*$ which is a supergraph of $G$ (or $G$ is a subgraph of $G^*$), under what circumstances one can extend the $H$-decomposition of $G$ into an $H^*$-decomposition of $G^*$? 
In other words, given an edge-coloring of $G$ (that can be considered as a decomposition when each color class induces a graph in $H$), is it possible to extend this coloring to an edge-coloring of $G^*$ so that each color class of $G^*$ induces a graph in $H^*$? 

In this direction, Hilton \cite{H2} found necessary and sufficient conditions for a decomposition of  $K_m$ to be embedded into a Hamiltonian decomposition of  $K_{m+n}$, which later was generalized by Nash-Williams \cite{Nash87}. Hilton and Rodger \cite{HR} considered the embedding of Hamiltonian decompositions for complete multipartite graphs. For embedding factorizations see \cite{HJRW, RW}, where the connectivity of the graphs in $H^*$ is one defining property. We also note that embedding problems first were studied for latin squares by M. Hall \cite{MHall45}. For extensions of Hall's theorem see \cite{LinRod92,  AndHil1, AndHil2}.

By the \textit {multiplicity} of a pair of vertices  $u,v$ of $G$, we mean the number of edges joining $u$ and $v$ in $G$.
In this paper $K(a_1,\ldots, a_p;\lambda,\mu )$ denotes a graph with $p$ parts, the $i^{th}$ part having size $a_i$, in which multiplicity of each pair of vertices in the same part (in different parts) is  $\lambda$ ($\mu $, respectively). When $a_1 = \ldots = a_p=a$, we denote $K(a_1,\ldots, a_p;\lambda,\mu )$ by $K(a^{(p)};\lambda,\mu )$. Let us say that an edge in $K(a^{(p)};\lambda,\mu )$ is pure (mixed, respectively) if its endpoints are in the same (different, respectively) part(s). If we replace each edge of $G$ by $\lambda$ multiple edges, then we denote the new graph by $\lambda G$. 

The graph $K(a_1,\ldots, a_p;\lambda,\mu )$ is of particular interest to statisticians, who consider group divisible designs with two associate classes, beginning over 50 years ago with the work of Bose and Shimamoto \cite{BoSh}. Decompositions of $K(a_1,\ldots, a_p;\lambda,\mu )$ into  cycles of length $m$ have been studied for small values of $m$ \cite{FuRod98, FuRod01, FuRodSar}. Recently, Bahmanian and Rodger  have settled the existence problem completely for longest (i.e. Hamiltonian) cycles in \cite{BahRod1}. In this paper, we study conditions under which one can embed a decomposition of $K(a^{(p)};\lambda,\mu )$ into a Hamiltonian decomposition of $K(a^{(p+r)};\lambda,\mu )$ for $r>0$. Our proof is largely based on our results in \cite{BahRod1} (see Theorem \ref{mainthesp} in the next section).

If $G$ is a $k$-edge-colored graph, and if $u,v\in V(G)$  then  $\ell(u)$ denotes the number of loops incident with vertex $u$, $d(u)$ denotes the degree of vertex $u$ (loops are considered to contribute two to the degree of the incident vertex), and $m(u,v)$ denotes the multiplicity of pair $u$  and $v$. The subgraph of $G$ induced by the edges colored $j$ is denoted by $G(j)$, and $\omega (G)$ is the number of components of $G$.

\section{Amalgamation and Graph Embedding} 
\textit{Amalgamating} a finite graph $G$ can be thought of as taking $G$, partitioning its vertices, then for each element of the partition squashing the vertices to form a single vertex in the amalgamated graph $H$. Any edge incident with an original vertex in $G$ is then incident with the corresponding new vertex in $H$, and any edge joining two vertices that are squashed together in $G$ becomes a loop on the new vertex in $H$. 

A \textit{detachment} of $H$ is, intuitively speaking, a graph $G$ obtained from $H$ by splitting some or all of its vertices into more than one vertex. That is, to each vertex $\alpha$ of $H$ there corresponds a subset $V_{\alpha}$ of $V(G)$ such that an edge joining two vertices $\alpha$ and $\beta$ in $H$ will join some element of  $V_{\alpha}$ to some element of $V_{\beta}$. 
If $\eta$ is a function from $V(H)$ into $\mathbb {N}$ (the set of positive integers), then an \textit{$\eta$-detachment} of $H$ is a detachment of $H$ in which each vertex $u$ of $H$ splits into $\eta(u)$ vertices. For a more precise definition of amalgamation and detachment, we refer the reader to \cite{BahRod1}. 

Since two graphs $G$ and $H$ related in the above manner have an obvious bijection between the edges, an edge-coloring of $G$ or $H$, naturally induces an edge-coloring on the other graph. Hence an amalgamation of a graph with colored edges is a graph with colored edges. 

The technique of vertex amalgamation, which was developed in the 1980s by Rodger and Hilton, has proved to be very powerful in decomposing of various classes of graphs. For a survey about the method of amalgamation and embedding partial edge-colorings we refer the reader to \cite{AndRod1}. In \cite{BahRod1}, the authors proved a general detachment theorem for multigraphs. For the purpose of this paper we use a very special case of this theorem as follows:

\begin{theorem}\textup{(Bahmanian, Rodger \cite[Theorem 1]{BahRod1})}\label{mainthesp}
Let $H$ be a $k$-edge-colored graph all of whose color classes are connected, and let $\eta$ be a function from $V(H)$ into $\mathbb{N}$ such that for each $v \in V(H)$\textup{:} \textup{(i)} $\eta (v) = 1$ implies $\ell_H (v) = 0$, \textup {(ii)} $d_{H(j)}(v)/\eta(v)$ is an even integer for $1\leq j\leq k$, \textup {(iii)} $\binom{\eta(v)}{2}$ divides $\ell_H(v)$, and \textup {(iv)} $\eta(v)\eta(w)$ divides $m_{H}(v,w)$ for each $w\in V(H)\backslash\{v\}$. Then there exists a loopless $\eta$-detachment $G$ of $H$ in which each $v\in V(H)$ is detached into $v_1,\ldots, v_{\eta(v)}$, all of whose color classes are connected, and for $v\in V(H)$\textup{:}
\begin{itemize}
\item [\textup{(i)}] $m_G(v_i, v_{i'}) = \ell_H(v)/\binom {\eta(v)}{2} $ for $1\leq i<i'\leq \eta(v)$ if $\eta (v) \geq 2$, 
\item [\textup{(ii)}] $m_G(v_i, w_{i'}) = m_H(v, w)/(\eta (v) \eta (w)) $ for $w\in V(H)\backslash\{v\}$, $1\leq i\leq \eta(v)$ and $1\leq i'\leq \eta(w)$, and
\item [\textup{(iii)}] $d_{G(j)}(v_i) = d_{H(j)}(v)/\eta (v)$ for $1\leq i\leq \eta(v)$  and $1\leq j\leq k$.
\end{itemize}
\end{theorem}

Here is our main result:
\begin{theorem} \label{embed1th}
Let $G=K(a^{(p)};\lambda,\mu )$ with $a>1$, $\lambda\geq 0$, $\mu \geq 1$, $\lambda\neq \mu $, $r\geq 1$, and let $\omega_j=\omega\big(G(j)\big)$. For $1\leq j\leq k$, define  
\begin{equation} \label{congcond}
s_j\equiv \omega_j \pmod r \mbox{ with }1\leq s_j\leq r,
\end{equation}
 and suppose 
\begin{equation} \label{ineqsumcond}
\sum_{j=1}^k s_j\geq kr-\mu a^2\binom{r}{2}.
\end{equation} 
Then a $k$-edge-coloring of $G$  can be embedded into a Hamiltonian decomposition of $G^*=K(a^{(p+r)};\lambda,\mu )$ if and only if: 
\begin{enumerate}
\renewcommand{\theenumi}{\textup{(\roman{enumi})}}
\renewcommand{\labelenumi}{\theenumi}
\item $k=\big(\lambda(a-1)+\mu a(p+r-1)\big)/2$,
\item $\lambda\leq \mu a(p+r-1)$,
\item  Every component of $G(j)$ is a path (possibly of length $0$) for $1\leq j\leq k$, and 
\item $\omega_j\leq ar$ for $1\leq j\leq k$.
\end{enumerate}
\end{theorem}
\begin{proof}
By Theorem 4.3 in \cite {BahRod1}, for $K(a^{(p+r)};\lambda,\mu )$ to be Hamiltonian decomposable, conditions (i) and (ii) are necessary and sufficient. (Condition (i) follows since $k$ must be $d_{G^*}(v)/2$. Condition (ii) follows since each Hamiltonian cycle must use at least $p+r$ mixed edges, so there must be sufficiently many mixed edges for all pure edges to be used.)   
For $1\leq j\leq k$, for $G(j)$ to be extendable to a Hamiltonian cycle in $K(a^{(p+r)};\lambda,\mu )$, the degree of each vertex has to be at most 2, and thus every component must be a path. Moreover, since each new vertex can link together two disjoint paths,
the number of components of every color class can not exceed the number of new vertices, $ar$. This proves the necessity of (i)--(iv).

Let $G=(V,E)$, and let $u$ be a vertex distinct from vertices in $V$. Define the new graph $G_1=(V_1,E_1)$ with $V_1=V\cup\{u\}$ by adding to $G$ the vertex $u$ incident with $\mu a^2\binom{r}{2}$ loops, and adding $\mu  ar$ edges between $u$ and each vertex $v$ in $V$ (see Figure \ref{figure:G_1G_2}). Note that for each $v\in V$, $d_{G_1}(v)=\lambda(a-1)+\mu a(p-1)+\mu ar=\lambda(a-1)+\mu a(p+r-1)=2k$. 
\begin{figure}[htbp]
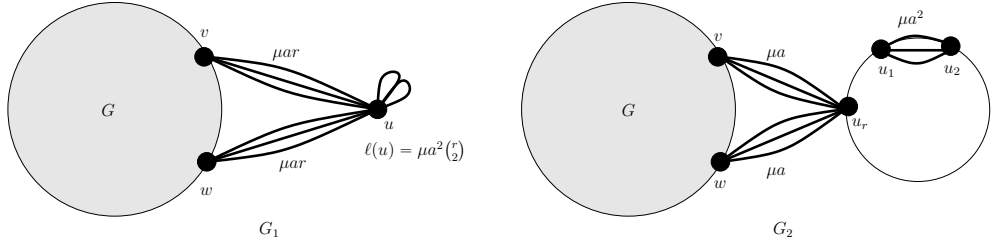

\begin{center}
\scalebox{.55}{ \input {G_1embedmu.pstex_t} \quad\quad\quad\quad\quad\quad   \input {G_2embedmu.pstex_t}}
\caption{$G_1$ and its detachment $G_2$}
\label{figure:G_1G_2}
\end{center}
\end{figure}
Now we extend the $k$-edge-coloring of $G$ to a $(k+1)$-edge-coloring of $G_1$ as follows:
\begin{enumerate}
\item [(A1)] Each edge in $G$ has the same color as it does in $G_1$,
\item [(A2)] For every $v\in V$, color the $\mu  ar$ edges between $v$ and $u$  so that $d_{G_1(j)}(v)=2$ for $1\leq j\leq k$. Since  $d_{G(j)}(v)\leq 2$ for $1\leq j\leq k$, and since $d_{G_1}(v)=2k$, this can be done. Notice that for every component of $G(j)$ (which is a path), exactly two edges (from end points of the path) are connected to $u$; so at this point $d_{G_1(j)}(u)=2\omega_j$ for $1\leq j\leq k$. 
\item [(A3)]  For $1\leq j\leq k$  color $r-s_j$ ($\geq 0$) loops with $j$. This coloring of loops can be done, since  by condition (2) of the theorem we have:
\begin{eqnarray*}
\sum_{j=1}^k s_j\geq  kr-\mu a^2\binom{r}{2} & \iff &  \sum_{j=1}^k r-\sum_{j=1}^k s_j\leq \mu a^2\binom{r}{2} \\
& \iff & \sum_{j=1}^k (r-s_j)\leq \mu a^2\binom{r}{2}=\ell_{G_1}(u).
\end{eqnarray*}
Moreover we color the remaining $\sum_{j=1}^k s_j-kr+\mu a^2\binom{r}{2}$ ($\geq 0$) loops with the new color $k+1$. 
Thus for $1\leq j\leq k$, $$d_{G_1(j)}(u)=2\omega_j+2(r-s_j)=2r+2(\omega_j-s_j),$$ and $d_{G_1(k+1)}(u)=2\Big(\sum_{j=1}^k s_j-kr+\mu a^2\binom{r}{2}\Big)$. By (1)  $d_{G_1(j)}(u)$  is an even multiple of $r$ for $1\leq j\leq k$. Now to show that $d_{G_1(k+1)}(u)$ is an even multiple of $r$, first we show that $\sum_{j=1}^{k}\omega_j=\mu a^2pr/2$. 
\begin{eqnarray*}
\sum \nolimits_{j=1}^k \omega_j  & = &  \sum \nolimits_{j=1}^k (pa-|E(G(j))|) \\
& = & kpa- |E|\\
&= & pa\big(\lambda(a-1)+\mu a(p+r-1)\big)/2-pa\big(\lambda(a-1)+\mu a(p-1)\big)/2 \\
&= &  \mu a^2pr/2. 
 \end{eqnarray*} 
Notice that $\mu a(p+r-1)$ is even, since otherwise, in particular $a$ would be odd, so $k$ would not be an integer. Thus, 
\begin{eqnarray*}
d_{G_1(k+1)}(u)  & \equiv   &  2\sum \nolimits_{j=1}^k \omega_j+\mu a^2r(r-1) \\
& \equiv & \mu a^2pr+\mu a^2r(r-1)\\
&\equiv &   \mu a^2r(p+r-1)\\
&\equiv &  0 \pmod {2r}.
 \end{eqnarray*}

Let $b_1,\ldots,b_{k+1}$ be even integers such that $d_{G_1(j)}(u)=b_jr$ for $1\leq j\leq k+1$. Note that for  $1\leq j\leq k$, we have
$$b_j/2=1+\frac{\omega_j-s_j}{r}\leq 1+\lfloor \frac{ar-1}{r}\rfloor\leq 1+(a-1)=a.$$
\end{enumerate}

Since each component of $G(j)$ is joined to $u$ in $G_1(j)$,  each color class of $G_1$ is connected. Let $\eta$ be a function from $V_1$ into $\mathbb N$ such that $\eta(v)=1$ for each $v\in V$, and $\eta(u)=r$. Now by Theorem \ref{mainthesp}, there exists an $\eta$-detachment $G_2$ of $G_1$, all of whose color classes are connected, (see Figure \ref{figure:G_1G_2}) in which $u$ is detached into $r$ new vertices $u_1,\ldots,u_r$ such that:
\begin{itemize}
\item [(a)] $m_{G_2}(u_i,u_{i'})=\mu a^2\binom{r}{2}/\binom{r}{2}=\mu a^2,$ for  $1\leq i<i'\leq r$;  
\item [(b)] $m_{G_2}(u_i,v)=\mu ar/r=\mu a$ for each $v\in V$ and each $i$, $1\leq i \leq r$; 
\item [(c)] $d_{G_2(j)}(u_i)=b_jr/r=b_j$ for $1\leq i \leq r$ and  $1\leq j \leq k+1$. \end{itemize}
We observe that $d_{G_2}(u_i)=ap(\mu a)+(r-1)\mu a^2=\mu a^2(p+r-1)$ for $1\leq i\leq r$, and is even. 
Note that by (c), $d_{G_2(j)}(u_i)=d_{G_2(j)}(u_{i'})$ and is even for $1\leq i\leq i' \leq r$, and we know that $d_{G_2}(u_i)\leq 2ka$ for $1\leq i \leq r$. Therefore, since $G(k+1)$ is an even graph, (so it has a 2-factorization), we can recolor each 2-factor of color class $k+1$ with a color $j$,  $1\leq j\leq  k$ such that $d_{G_2(j)}(u_i)\leq 2a$. 
We let $b'_1,\ldots,b'_{k}$ be even integers such that in the resulting edge-coloring of $G_2$ obtained from recoloring the color class $k+1$, $d_{G_2(j)}(u)=b'_jr$ for $1\leq j\leq k$.

Now we define the new graph $G_3$ by adding $\lambda\binom{a}{2}$ loops on every vertex $u_i$ in $G_2$, for $1\leq i \leq r$ (see Figure \ref{figure:G_3}). 
\begin{figure}[htbp]
\begin{center}
\scalebox{.55}{  \input {G_3embedmu.pstex_t}}
\caption{$G_3$}
\label{figure:G_3}
\end{center}
\end{figure}
We extend the $k$-edge-coloring of $G_2$ to a $k$-edge-coloring of $G_3$ such that:
\begin{enumerate}
\item [(B1)] Each edge in $G_2$ has the same color at it does in $G_3$,
\item [(B2)] For $1\leq i\leq r$ and $1\leq j\leq k$, there are $a-b'_j/2$ ($\geq 0$) loops incident with $u_i$ colored $j$. This is possible, for the following reason: 
\begin{eqnarray*}
\sum_{j=1}^k (a-b'_j/2)  & = &  ka-\frac{1}{2}\sum_{j=1}^k d_{G_2(j)}(u_1) \\
& = & ka- \frac{1}{2} d_{G_2}(u_1) \\
&= & ka-\frac{1}{2}\mu a^2(p+r-1)\\
&= & \frac{a}{2}\big( \lambda(a-1)+\mu a(p+r-1) \big)-\frac{1}{2}\mu a^2(p+r-1)\\
&= &  \frac{a}{2} \lambda(a-1)= \ell_{G_3}(u_1). 
 \end{eqnarray*} 
\end{enumerate}
Since for $1\leq j\leq k$, $G_2(j)$ is a connected spanning subgraph of $G_3(j)$, $G_3(j)$ is also connected. Let $\eta'$ be a function from $V_3$ into $\mathbb N$ such that $\eta'(v)=1$ for each $v\in V$, and $\eta'(u_i)=a$ for $1\leq i\leq r$. Now by Theorem \ref{mainthesp}, there exists an $\eta'$-detachment $G_4$ of $G_3$, all of whose color classes are connected, in which $u_i$ is detached into $a$ new vertices $u_{i1},\ldots,u_{ia}$ for $1\leq i\leq r$ such that:
\begin{itemize}
\item  $m_{G_4}(u_{ij},u_{ij'})=\lambda\binom{a}{2}/\binom{a}{2}=\lambda$ for $1\leq i\leq r$ and $1\leq j<j'\leq a$; 
\item $m_{G_4}(u_{ij},u_{i'j'})=\mu a^2/a^2=\mu $ for $1\leq i<i'\leq r$  and $1\leq j<j'\leq a$; 
\item $m_{G_4}(u_{ij},v)=\mu a/a=\mu $ for each $v\in V$ and for $1\leq i \leq r$; and 
\item $d_{G_4(j)}(u_{ii'})=2a/a=2$ for $1\leq i \leq r$, $1\leq i' \leq a$. 
\end{itemize}
Therefore $G_4=K(a^{(p+r)};\lambda,\mu )$, and each color class in $G_4$ is a Hamiltonian cycle, so the proof is complete.
\end{proof}

 A natural perspective of this embedding problem is to keep $a, p, \lambda$ and $\mu$ fixed, and ask for which values of $r$ the embedding is possible. The following result completely settles this question for all $r\geq  \frac{\lambda(a-1)+\mu a(p-1)}{\mu a(a-1)}$.
\begin{theorem} \label{embed1thasym}
Let $G=K(a^{(p)};\lambda,\mu )$ with $a>1$, $\lambda\geq 0$, $\mu \geq 1$, $\lambda\neq \mu $, and
\begin{equation}\label{rlargewrtp}
r\geq  \frac{\lambda(a-1)+\mu a(p-1)}{\mu a(a-1)}.
\end{equation}
Then a $k$-edge-coloring of $G$  can be embedded into a Hamiltonian decomposition of $K(a^{(p+r)};\lambda,\mu )$ if and only if \textup {(i)--(iv)} of \textup {Theorem \ref{embed1th}} are satisfied.
 \end{theorem}
\begin{proof}
It is enough to show that (\ref{rlargewrtp}) implies (\ref{ineqsumcond}). Since $s_j\geq 1$ for $1\leq j\leq k$, $\sum_{j=1}^k s_j \geq k$. Thus, if we show that $k\geq kr-\mu a^2\binom{r}{2}$, we are done. This is true by the following sequence of equivalences:
\begin{eqnarray*}
k(r-1)\leq \mu a^2\binom{r}{2} & \iff &   \\
(r-1)\big(\lambda (a-1)+\mu a (p+r-1)\big) \leq \mu a^2 r(r-1) & \iff & \\
\lambda (a-1) \leq \mu a(ar-p-r+1)=\mu a\big( r(a-1)-(p-1)  \big) & \iff & \\
\lambda (a-1)/(\mu a) \leq r(a-1)-(p-1)& \iff & \\
r\geq  \frac{\lambda(a-1)+\mu a(p-1)}{\mu a(a-1)}.
 \end{eqnarray*} 
 \end{proof}
Another immediate corollary of Theorem \ref{embed1th} is the following complete solution to the embedding problem when $r=1$:
\begin{corollary} \label{embedcorr=1}
Let $G=K(a^{(p)};\lambda,\mu )$ with $a>1$, $\lambda\geq 0$, $\mu \geq 1$, $\lambda\neq \mu $. Then a $k$-edge-coloring of $G$  can be embedded into a Hamiltonian decomposition of $K(a^{(p+1)};\lambda,\mu )$ if and only if: 
\begin{enumerate}
\renewcommand{\theenumi}{\textup{(\roman{enumi})}}
\renewcommand{\labelenumi}{\theenumi}
\item $k=\big(\lambda(a-1)+\mu ap\big)/2$,
\item $\lambda\leq \mu ap$,
\item  Every component of $G(j)$ is a path (possibly of length $0$) for $1\leq j\leq k$, and 
\item $\omega_j\leq a$ for $1\leq j\leq k$.
\end{enumerate}
\end{corollary}
\begin{proof}
Since $r=1$, we have $s_1=\ldots=s_k=1$, so $k=\sum_{j=1}^k s_j= k-\mu a^2\binom{1}{2}=k$, and thus condition (\ref{ineqsumcond}) of Theorem \ref{embed1th} is satisfied. 
\end{proof}
\begin{proposition}
Whenever $\lambda \leq \mu a$ and $p\leq a$, the embedding problem is completely solved for all values of $r\geq 1$. 
\end{proposition}
\begin{proof}
Condition \ref{rlargewrtp} can be rewritten as $r\geq \frac{\lambda}{\mu a}+\frac{p-1}{a-1}$. Since we are assuming that $\lambda \leq \mu a$ and $p\leq a$, we have $\frac{\lambda}{\mu a}+\frac{p-1}{a-1}\leq 2$. Therefore the result follows from 
Theorem \ref{embed1thasym} and  Corollary \ref{embedcorr=1}. 
\end{proof}
\begin{example}
\textup{
A $k$-edge-coloring of $K(10^{(7)}; 2, 5)$ can be embedded into a Hamiltonian decomposition of $K(10^{(7+r)}; 2, 5)$ for $r\geq 1$ if and only if \textup {(i)--(iv)} of \textup {Theorem \ref{embed1th}} are satisfied.  
}
\end{example}
The following result completely settles the embedding problem for the smallest value of $r$ in another sense, namely with respect to the inequality (ii) of Theorem \ref{embed1th}; so it settles the case where $\lambda=\mu a(p+r-1)$, or equivalently where $r=\frac{\lambda}{\mu a}-p+1$. 
The proof is similar to that of Theorem \ref{embed1th}, so only an outline of the proof is provided, the details being omitted. The proof of the necessity of condition (ii) of Theorem \ref{embed1th} shows that, in a Hamiltonian decomposition of $K(a^{(p)};\lambda a(p+r-1),\lambda)$, each Hamiltonian cycle contains exactly $a-1$ pure edges from each part, and exactly $p+r$ mixed edges.

\begin{theorem}
Let $a>1$ and $r,\mu \geq 1$. A $k$-edge-coloring of $G=K(a^{(p)};\mu  a(p+r-1),\mu )$ can be embedded into a Hamiltonian decomposition of $G^*=K(a^{(p+r)};\mu  a(p+r-1),\mu )$ if and only if: 
\begin{itemize}
\item [\textup{(i)}] $k=\mu  a^2(p+r-1)/2$,
\item [\textup{(ii)}]   Every component of $G(j)$ is a path (possibly of length $0$) for $1\leq j\leq k$, 
\item [\textup{(iii)}] $G(j)$ has exactly $a-1$ pure edges from each part, and at most $p-1$ mixed edges for $1\leq j \leq k$, and
\item  [\textup{(iv)}] $\omega_j\leq r$ for $1\leq j \leq k$.
\end{itemize}
\end{theorem}
\begin{proof}
The necessity of (i)--(iii) follows as described in Theorem \ref{embed1th}. Let $m_j$ be the number of mixed edges in $G(j)$. To extend each component $P$ of $G(j)$ to a Hamiltonian cycle in $G^*$, two mixed edges have to join $P$ to the new vertices, and since each Hamiltonian cycle in $G^*$ contains exactly $p+r$ mixed edges, we have that 
\begin{equation}\label{1eqa}
m_j+2\omega_j\leq p+r.
\end{equation}
 Since $G(j)$ is a collection of paths, for $1\leq j\leq k$, we have $|V(G(j))|=|E(G(j))|+\omega_j$. Therefore $ap=m_j+p(a-1)+\omega_j$ and thus
\begin{equation}\label{2eqa}
m_j+\omega_j=p.
\end{equation}
Combining (\ref{1eqa}) and (\ref{2eqa}) implies (iv). 

To prove sufficiency, we define the graph $G_1$ as it is defined in Theorem \ref{embed1th}. We extend the $k$-edge-coloring of $G$ to a $k$-edge-coloring of $G_1$ such that $d_{G_1(j)}(v)=2$ for each $v\in V$ and $1\leq j\leq k$. This is possible by the same argument as in Theorem \ref{embed1th}. At this point $d_{G_1(j)}(u)=2\omega_j\leq 2r$ for $1\leq j\leq k$. So we can color the loops incident with $u$ such that $d_{G_1(j)}(u)=2r$ for $1\leq j\leq k$, simply by assigning the color $j$ to $r-\omega_j$ loops. 

Now we detach the vertex $u$ into $r$ new vertices $u_1,\ldots,u_r$ to obtain the new graph $G_2$ (as we did in the proof of Theorem \ref{embed1th}). Note that $d_{G_2(j)}(u_i)=2r/r=2$ for each $i$, $1\leq i \leq r$ and each $j$, $1\leq j \leq k$. 
Now we define the new graph $G_3$ by adding $a-1$ loops of color $j$, $1\leq j\leq k$, on every vertex $u_i$ in $G_2$, for each $i$, $1\leq i \leq r$. 
So we have $d_{G_3(j)}(u_i)=2a$. Using Theorem \ref{mainthesp}, detach each vertex $u_i$  into $a$ new vertices $u_{i1},\ldots,u_{ia}$ for $1\leq i\leq r$, to obtain the new graph $G_4$ in which, $G_4(j)$ is connected and $d_{G_4(j)}(u_{ii'})=2a/a=2$ for $1\leq j\leq k$, $1\leq i \leq r$, $1\leq i' \leq a$. This completes the proof. 
\end{proof}

\end{document}